\documentclass[12pt]{amsart}

\usepackage{amssymb}
\usepackage{amsmath}
\usepackage{dsfont}
\usepackage{url}
\usepackage{enumerate}

\usepackage{graphicx}

\makeatletter
\@namedef{subjclassname@2010}{%
  \textup{2010} Mathematics Subject Classification}
\makeatother

\newtheorem{thm}{Theorem}[section]

\newtheorem{lem}[thm]{Lemma}
\newtheorem{pro}[thm]{Proposition}

\theoremstyle{definition}

\numberwithin{equation}{section}

\newcommand{\sgn}{\textrm{sgn}}

\newcommand{\Z}{\mathbb{Z}}

\newcommand{\N}{\Z_{>0}}
\newcommand{\bigoh}{{\mathcal{O}}}
\newcommand{\e}{{\rm e}}

\begin{document}


\baselineskip=17pt



\title{Explicit Upper Bounds for $\left|L(1, \chi)\right|$ when $\chi(3)=0$}

\author[S. SAAD EDDIN]{Sumaia SAAD EDDIN}
\address{Laboratoire Paul Painlev\'e,
Universit\'e des Sciences et Technologies de Lille,
B\^atiment M2,
cit\'e Scientifique,
59655 Villeneuve d'ascq Cedex, France.}
\email{sumaia.saad-eddin@math.univ-lille1.fr}

\author{David J. Platt}
\address{Heilbronn Institute for Mathematical Research, University of Bristol, University Walk, Bristol, BS8 1TW, United Kingdom}
\email{Dave.Platt@bristol.ac.uk}

\date{24 May 2013}

\begin{abstract}
Let $\chi$ be a primitive Dirichlet character of conductor $q$ and let us denote by $L(z, \chi)$ the associated L-series.
In this paper we provide an explicit upper bound for $\left|L(1, \chi)\right|$ when $3$ divides $q$.
\end{abstract}

\subjclass[2010]{Primary 11M06; Secondary 11Y35}

\keywords{Dirichlet characters, Dirichlet $L$-function, Gauss sums.}

\maketitle

\section{Introduction and results}
Let $\chi$ be a primitive Dirichlet character of conductor $q$ and let us denote by $L(z, \chi)$ the associated L-series. Recall we describe $\chi$ as even when $\chi(-1)=1$ and odd when $\chi(-1)=-1$. The upper bound for $\left|L(1, \chi)\right|$ has received considerable attention, especially in the past $25$ years. In \cite{Louboutin2002} and \cite{Louboutin2004}, Louboutin by used integral representations of Dirichlet $L$ functions to obtain the following upper bound for $\left|L(1, \chi)\right|$ when $3$ divides $q$; 
\begin{equation*}
|L(1,\chi)|\leq \tfrac 13\log q + \left\{ \begin{array}{ll}
        0.3816
        & \textrm{for even $\chi$,}\vspace{0.5cm}\\
        0.8436
        & \textrm{for odd $\chi$.}
\end{array} \right.
\end{equation*}
In 2001, Ramar\'e \cite{Ramare2001} described a different method (see Proposition~\ref{pro1}). In this paper, we exploit Ramar\'e's method to improve on Louboutin's bound. Specifically, we show;
\begin{thm}
\label{Thm4}
Let $\chi $ be a primitive character of conductor $q$ such that $3|q$. Then 
\begin{equation*}
|L(1,\chi)|\leq \tfrac 13\log q + \left\{ \begin{array}{ll}
        0.368296
        & \textrm{for even $\chi$,}\vspace{0.5cm}\\
        0.838374
        & \textrm{for odd $\chi$.}
\end{array} \right.
\end{equation*}
\end{thm}
In section~\ref{sec3} we establish this result for conductor $q~>~2\cdot~10^6$. In section~\ref{sec4} we describe a rigorous computation to confirm that the result also holds for conductor $q\in [3,2\cdot 10^6]$.
\section{Lemmas from elsewhere}
\label{sec1}
The proof of the following results can be found in \cite{Ramare2001}. We use $\tau(\chi)$ to denote the Gauss sum
\begin{equation*}
\tau(\chi):=\sum\limits_{a\;\textrm{mod}\;q} \chi(a) e\left(a/q\right)
\end{equation*}
where, as usual
\begin{equation*}
\e(x):=\exp(2\pi i x).
\end{equation*}
The modulus of the Gauss sum is well known when $\chi$ is primitive and given by $|\tau(\chi)|=\sqrt{q}$.

\begin{pro}
\label{pro1}
Set
\begin{eqnarray*}
 F_3(t)&=&\left(\frac{\sin (\pi t)}{\pi}\right)^2\left(\frac{2}{t}+\sum\limits_{m\in \mathbb{Z}}\frac{\sgn (m)}{(t-m)^2}\right),
\\&&
j(t)=2\int\limits_{|t|}^{1}\left(\pi (1-u)\cot (\pi u)+1\right)\, du,
\\&&
F_4(t)=1-\left(\frac{\sin (\pi t)}{\pi t}\right)^2.
\end{eqnarray*}
Let $\chi$ be a primitive Dirichlet character of conductor $q$. Then, we have: 
\begin{equation*}
L(1, \chi)=
\begin{cases}
    \displaystyle{\sum\limits_{n\geq 1} \frac{\left(1-F_3(\delta n)\right)\chi(n)}{n}-\frac{\tau(\chi)}{q}\sum\limits_{1\leq m\leq \delta q}\overline{\chi }(m)j\left(\frac{m}{\delta q}\right) } &\textrm{for even $\chi$,}\\
  
  \displaystyle{\sum\limits_{n\geq 1} \frac{\left(1-F_4(\delta n)\right)\chi(n)}{n}+\frac{i\pi\tau(\chi)}{q}\sum\limits_{1\leq m\leq \delta q}\overline{\chi }(m)\left(1-\frac{m}{\delta q}\right)^2 } &\textrm{for odd $\chi$.}
\end{cases}
\end{equation*}
\end{pro}

\begin{lem}
\label{ramare4}
Set 
\begin{equation*}
j(t)=2\int\limits_{|t|}^{1}\left(\pi (1-u)\cot (\pi u)+1\right)\, du.
\end{equation*}
Then, we have 
\begin{equation}
\label{eq247}
\int_{0}^{1}j(t)\, dt =1
\end{equation}
and
\begin{equation}
\label{cor13}
-2\log |t|-2\left(\log (2\pi)-1\right)\leq j(t)\leq -2\log |t|.
\end{equation}
\end{lem}
\begin{lem}
\label{ramare2}
For $\delta \in ]0, 1 ]$, we have
\begin{equation*}
\sum\limits_{n\geq 1}\frac{1-F_3(\delta n)}{n}=-\log \delta -1+\delta.
\end{equation*}
where $F_3$ is defined in Proposition~\ref{pro1}.
\end{lem}
\begin{lem}
\label{ramare5}
For $\delta >0$, we have
\begin{equation*}
\sum\limits_{n\geq 1}\frac{1-F_4(\delta n)}{n}=-\log \delta +\frac{3}{2}-\log (2\pi)+2\int\limits_{0}^{1}(1-t)\log \left|\frac{\pi \delta t}{\sin (\pi \delta t)}\right|\, dt.
\end{equation*}
where $F_4$ is defined in Proposition~\ref{pro1}.
\end{lem}
\begin{lem}
\label{ramare31}
Let $k$ and $\beta>0$ be two real numbers. Let $f$ be a continuous, convex and non-increasing $\mathds{L}^1$-function on $]k-\theta,\, k+\beta]$. We have 
\begin{equation*}
f(k)\leq \frac{1}{\theta}\int_{k-\theta}^{k}f(t)\, dt -\theta\frac{f(k)-f(k+\beta)}{2\beta}.
\end{equation*}
\end{lem}
\begin{lem}
\label{ramare6}
For $0<\delta \leq 1/2$, we have 
\begin{equation}
2\int_{0}^{1}\left(1-t\right)\log \left|\frac{\pi \delta t}{\sin (\pi \delta t)}\right|\, dt
\leq \frac{\pi^3 \delta^2}{12}.
\end{equation}
\end{lem}
\section{Some auxiliary lemmas}
\label{sec2}
\begin{lem}
\label{lem41}
Let $\alpha>0$ and $a$ be $1$ or $2$. Let $g$ be a continuous, convex, non-negative and non-increasing $\mathds{L}^1$-function on $]0,1]$. Then 
\begin{equation*}
\sum_{\substack{ a\leq n\leq \alpha \\ n=3k+a}}g\left(\frac{n}{\alpha}\right)
\leq
        \frac{1}{3}\int_{0}^{\alpha}g\left(\frac{u}{\alpha}\right)\, du
         +\frac{1}{2}\, g(1)
        +\frac{2-a}{2}\, g\left(\frac{a}{\alpha}\right)
        +\frac{a}{6}\, g\left(\frac{2a}{\alpha}\right)
        -\frac{1}{2}\, g\left(\frac{a+3}{\alpha}\right)
\end{equation*}
\end{lem}
\begin{proof}
Let $3K+a$ be the largest integer less than or equal to $\alpha$, then we have 
\begin{equation}
\label{eq243}
\sum_{\substack{ a\leq n\leq 3K+a \\n=3k+a}}g\left(\frac{n}{\alpha}\right)
       =
       g\left(\frac{a}{\alpha}\right)
       +
        \sum_{\substack{ a+3\leq n\leq 3K+a\\n=3k+a}}g\left(\frac{n}{\alpha}\right).
\end{equation}  
We write the sum on the right side-hand of this equality as
\begin{equation*}
       \sum_{\substack{a+3\leq n\leq 3K+a \\n=3k+a}}g\left(\frac{n}{\alpha}\right)
=
       \sum_{1\leq k\leq K}g\left(\frac{3k+a}{\alpha}\right).
\end{equation*}
Using Lemma~\ref{ramare31} with $\theta=\beta=1$, we get  
\begin{eqnarray*}
\sum_{k=1}^{K-1}g\left(\frac{3k+a}{\alpha}\right)
&\leq&
 \sum_{k=1}^{K-1}\int_{k-1}^{k}g\left(\frac{3t+a}{\alpha}\right)\, dt-\sum_{k=1}^{K-1}\frac{g\left(\frac{3k+a}{\alpha}\right)-g\left(\frac{3k+a+3}{\alpha}\right)}{2},
\\&\leq&
        \int_{0}^{K-1}g\left(\frac{3t+a}{\alpha}\right)\, dt
  -
        \frac{g\left(\frac{a+3}{\alpha}\right)-g\left(\frac{3K+a}{\alpha}\right)}{2}.
\end{eqnarray*}
By making the simple change of variable $3t+a=u$ in the integral above, we obtain 
\begin{equation}
\label{eq242}
        \sum_{k=1}^{K-1}g\left(\frac{3k+a}{\alpha}\right)
\leq
        \frac{1}{3}\int_{a}^{3K+a-3}g\left(\frac{u}{\alpha}\right)\, du
 -
        \frac{g\left(\frac{a+3}{\alpha}\right)-g\left(\frac{3K+a}{\alpha}\right)}{2}.
\end{equation}
We again use Lemma~\ref{ramare31} but this time with $\theta=3$ and $\beta=\alpha-(3K+a)$. When $\beta=0$ the proof is complete. Otherwise $\beta \leq 3$ and we get
\begin{equation*}
        g\left(\frac{3K+a}{\alpha}\right)
\leq 
        \frac{1}{3}\int_{3K+a-3}^{3K+a}g\left(\frac{u}{\alpha}\right)\, du 
 -
        \frac{3\left(g\left(\frac{3K+a}{\alpha}\right)-g(1)\right)}{2\beta}.
\end{equation*}
Substituting this in Eq~\eqref{eq242} with $\beta\leq 3$, we find that
\begin{equation*}
        \sum_{k=1}^Kg\left(\frac{3k+a}{\alpha}\right)
\leq 
        \frac{1}{3}\int_{a}^{3K+a}g\left(\frac{u}{\alpha}\right)\, du -\frac{g(\frac{a+3}{\alpha})-g(1)}{2}.
\end{equation*}
Then Eq~\eqref{eq243} becomes
\begin{equation}
\label{eq244}
         \sum_{\substack {a\leq n\leq 3K+a\\n=3k+a}}g\left(\frac{n}{\alpha}\right)
\leq
         \frac{1}{3}\int_{a}^{3K+a}g\left(\frac{u}{\alpha}\right)\, du 
  -
         \frac{g(\frac{a+3}{\alpha})-g(1)}{2}
  +
         g\left(\frac{a}{\alpha}\right).
\end{equation}
Now, we apply Lemma~\eqref{ramare31} to $g(a/\alpha)$ with $\theta=\beta=a$, to get
\begin{equation*}
\label{equ82}
         g\left(\frac{a}{\alpha}\right)\leq \frac{1}{a}\int_{0}^{a}g\left(\frac{u}{\alpha}\right)\, du
   -
         \frac{g(a/\alpha)-g(2a/\alpha)}{2},
\end{equation*}
it follows that
\begin{equation*}
  \frac{3}{2}g\left(\frac{a}{\alpha}\right)
       \leq 
  \frac{1}{a}\int_{0}^{a}g\left(\frac{u}{\alpha}\right)\, du +\frac{g(2a/\alpha)}{2}.
\end{equation*}
Multiplying this with $a/3$, we find that  
\begin{equation*}
  \frac{a}{2}\, g\left(\frac{a}{\alpha}\right)
       \leq 
  \frac{1}{3}\int_{0}^{a}g\left(\frac{u}{\alpha}\right)\, du +\frac{a}{6}\, g\left(\frac{2a}{\alpha}\right).
\end{equation*}
Here, we have to distinguish two cases.
\begin{itemize}
\item The first case when $a=1$, and write 
\begin{equation*}
  \frac{g(1 /\alpha)}{2}
       \leq 
  \frac{1}{3}\int_{0}^{1}g\left(\frac{u}{\alpha}\right)\, du +\frac{g(2/\alpha)}{6}.
\end{equation*}
Then, Eq~\eqref{eq244} becomes
\begin{equation}
\label{eq245}
       \sum_{\substack{ 1\leq n\leq 3K+1\\n=3k+a}}g\left(\frac{n}{\alpha}\right)
\leq 
        \frac{1}{3}\int_{0}^{3K+1}g\left(\frac{u}{\alpha}\right)\, du 
          +\frac{g(1)}{2}
        +\frac{g(1/\alpha)}{2}
        +\frac{g(2/\alpha)}{6}
        -\frac{g(4/\alpha)}{2}.
\end{equation}
\item The second case when $a=2$, we rewrite
 \begin{equation*}
  g\left(\frac{2}{\alpha}\right)
       \leq 
  \frac{1}{3}\int_{0}^{2}g\left(\frac{u}{\alpha}\right)\, du +\frac{g(4/\alpha)}{3}.
\end{equation*}
Then, Eq~\eqref{eq244} becomes 
\begin{equation}
\label{eq246}
       \sum_{\substack{2\leq n\leq 3K+2\\n=3k+a}}g\left(\frac{n}{\alpha}\right)
\leq 
        \frac{1}{3}\int_{0}^{3K+2}g\left(\frac{u}{\alpha}\right)\, du 
          +\frac{g(1)}{2}
        +\frac{g(4/\alpha)}{3}
        -\frac{g(5/\alpha)}{2}
\end{equation}
\end{itemize}
From Eq~\eqref{eq245} and Eq~\eqref{eq246}, we conclude that  
\begin{equation*}
       \sum_{\substack{a\leq n\leq \alpha \\n=3k+a}} g\left(\frac{n}{\alpha}\right)
\leq
        \frac{1}{3}\int_{0}^{3K+a}g\left(\frac{u}{\alpha}\right)\, du 
        +\frac{1}{2}\, g(1)
        +\frac{2-a}{2}\, g\left(\frac{a}{\alpha}\right)
        +\frac{a}{6}\, g\left(\frac{2a}{\alpha}\right)
        -\frac{1}{2}\, g\left(\frac{a+3}{\alpha}\right).
\end{equation*}
We complete the integral from $(3K+a)$ to $\alpha$ by using the non-negativity of $g$ and get
\begin{equation*}
       \sum_{\substack{ a\leq n\leq 3K+a\\n=3k+a}} g\left(\frac{n}{\alpha}\right)
\leq
        \frac{1}{3}\int_{0}^{\alpha}g\left(\frac{u}{\alpha}\right)\, du 
        +\frac{1}{2}\, g(1)
        +\frac{2-a}{2}\, g\left(\frac{a}{\alpha}\right)
        +\frac{a}{6}\, g\left(\frac{2a}{\alpha}\right)
        -\frac{1}{2}\, g\left(\frac{a+3}{\alpha}\right).
\end{equation*}
This completes the proof.
\end{proof}
\begin{lem}
\label{lem42}
For $0<\delta \leq 1/2$, we have
\begin{equation*}
\int_{0}^{1}\left(1-t\right)\left(\log \left|\frac{\pi \delta t}{\sin (\pi \delta t)}\right|-\frac{1}{3}\log \left|\frac{3\pi \delta t}{\sin (3\pi \delta t)}\right|\right)\, dt\leq \frac{\pi^3\delta ^2}{36}+\frac{\pi^2\delta ^2}{27}.
\end{equation*}
\end{lem}
\begin{proof}
We begin with 
\begin{equation*}
\log \left|\frac{3\pi \delta t}{\sin (3\pi \delta t)}\right|= \log 3+\log (\pi \delta t)-\log \left|\sin (3\pi \delta t)\right|.
\end{equation*}
Since 
\begin{equation*}
\sin (3x) = -4\sin^3 x+3\sin x,
\end{equation*}
it follows that 
\begin{eqnarray*}
\log \left|\frac{3\pi \delta t}{\sin (3\pi \delta t)}\right|&=&
 \log 3+\log (\pi \delta t)-\log \left|\sin (\pi \delta t)\right|-\log \left|4\sin^2 (\pi \delta t)-3\right|
 \\&=&
 \log \frac{\pi\delta t}{\left|\sin (\pi\delta t)\right|}-\log \left|\frac{4}{3}\sin^2 (\pi \delta t)-1\right|.
\end{eqnarray*}
Then
\begin{multline}
\label{equ83}
\int_{0}^{1}\left(1-t\right)\left(\log \left|\frac{\pi \delta t}{\sin (\pi \delta t)}\right|-\frac{1}{3}\log \left|\frac{3\pi \delta t}{\sin (3\pi \delta t)}\right|\right)\, dt=\\ \frac{2}{3}\int_{0}^{1}\left(1-t\right)\log \left|\frac{\pi \delta t}{\sin (\pi \delta t)}\right|\, dt
+\frac{1}{3}\int_{0}^{1}(1-t)\log \left|\frac{4}{3}\sin^2 (\pi \delta t)-1\right|\, dt.
\end{multline}
Since $\log x\leq x-1$ and $\sin x\leq x$ if $0\leq x\leq \pi/2$, we see that 
\begin{eqnarray*}
\log \left|\frac{4}{3}\sin^2 (\pi \delta t)-1\right|
&\leq& 
\left|\frac{4}{3}\sin^2 (\pi \delta t)-1\right|-1
\\&\leq &
\frac{4}{3}\sin^2 (\pi \delta t)\leq \frac{4}{3}(\pi\delta t)^2.
\end{eqnarray*}
Hence, we can write Eq~\eqref{equ83} as  
\begin{multline*}
\int_{0}^{1}\left(1-t\right)\left(\log \left|\frac{\pi \delta t}{\sin (\pi \delta t)}\right|-\frac{1}{3}\log \left|\frac{3\pi \delta t}{\sin (3\pi \delta t)}\right|\right)\, dt
\leq 
\\
 \frac{2}{3}\int_{0}^{1}\left(1-t\right)\log \left|\frac{\pi \delta t}{\sin (\pi \delta t)}\right|\, dt
+\frac{4(\pi\delta)^2}{9}\int_{0}^{1}(1-t)t^2\, dt. 
\end{multline*}
Using Lemma~\ref{ramare6}, we infer
\begin{equation*}
\int_{0}^{1}\left(1-t\right)\left(\log \left|\frac{\pi \delta t}{\sin (\pi \delta t)}\right|-\frac{1}{3}\log \left|\frac{3\pi \delta t}{\sin (3\pi \delta t)}\right|\right)\, dt
\leq \frac{\pi^3\delta ^2}{36}+\frac{\pi^2\delta ^2}{27}
\end{equation*}
and this completes the proof.
\end{proof}
\section[Proof of the main result]{Proof of Theorem~\ref{Thm4}}
\label{sec3}
We are now ready to prove our upper bound for $\left|L(1, \chi)\right|$ when $3$ divides the conductor $q$. We break the proof into two cases:
\begin{itemize}
\item For even characters, we have:
\begin{equation*}
\sum_{\substack{n\geq 1\\(n,3)=1}}
       \frac{1-F_3(\delta n)}{n}=\sum_{n\geq 1}\frac{1-F_3(\delta n)}{n}
       -\sum_{n\geq 1}\frac{1-F_3(3\delta n)}{3n}
\end{equation*}
where $F_{3}$ is defined by Proposition~\ref{pro1}. Thanks to  Lemma~\ref{ramare2}, we can write the sums in the right-hand side above as 
\begin{equation*}
\sum_{n\geq 1}\frac{1-F_{3}(\delta n)}{n}=-\log \delta-1+\delta
\end{equation*}
and 
\begin{equation*}
\sum_{n\geq 1}\frac{1-F_{3}(3\delta n)}{n}=-\log (3\delta)-1+3\delta 
\end{equation*}
respectively.
Again using Proposition~\ref{pro1} and recalling $0<\delta \leq 1$, we get 
\begin{equation}
\label{cor11}
\left|L(1, \chi)\right|\leq 
        -\frac{2}{3}\log \delta 
        +\frac{1}{3}\log 3-\frac{2}{3}
        +\frac{1}{\sqrt{q}}\sum_{\substack{1\leq m\leq \delta q \\(m,3)=1}}j\left(\frac{m}{\delta q}\right)
\end{equation}
where $j(t)$ is defined in Lemma~\ref{ramare4}.  
Now, we apply Lemmas~\ref{ramare4} and \ref{lem41} of the sum of $j(t)$ given in Eq~\eqref{cor11} to obtain
\begin{eqnarray*}
        \sum_{\substack{ 1\leq m\leq \delta q \\ (m,3)=1}}j\left(\frac{m}{\delta q}\right)
&=&
        \sum_{\substack{1\leq m\leq \delta q\\m=3k+1}}j\left(\frac{m}{\delta q}\right)
+
        \sum_{\substack{2\leq m\leq \delta q\\m=3k+2 }}j\left(\frac{m}{\delta q}\right)
\\&\leq&     
         \frac{2\delta q}{3}
          +j(1)
          +\frac{1}{2}\, j\left(\frac{1}{\delta q}\right)
          +\frac{1}{6}\, j\left(\frac{2}{\delta q}\right)
          -\frac{1}{6}\, j\left(\frac{4}{\delta q}\right)
          -\frac{1}{2}\, j\left(\frac{5}{\delta q}\right).
\end{eqnarray*} 
Using Lemma~\ref{ramare4} again, we find that 
\begin{equation*}
        \sum_{\substack{1\leq m\leq \delta q \\ (m,3)=1}}j\left(\frac{m}{\delta q}\right)
\leq     
         \frac{2\delta q}{3}
          +\frac{5}{3}\log 2
          +\log 5
          +\frac{4}{3}\left(\log \pi -1\right).
\end{equation*} 
Then, Eq~\eqref{cor11} becomes simply
\begin{equation*}
\left|L(1, \chi)\right|\leq 
        -\frac{2}{3}\log \delta 
        +\frac{1}{3}\log 3-\frac{2}{3}
        +\frac{2\delta\sqrt{q}}{3}
        +\frac{1}{\sqrt{q}}\left( 
          \frac{5}{3}\log 2
          +\log 5
          +\frac{4}{3}\log \pi -\frac{4}{3}\right).
\end{equation*}
The best possible choice for $\delta$ is $1/\sqrt{q}$. This yields
\begin{equation}
\label{equ84}
\left|L(1, \chi)\right|\leq 
        \frac{1}{3}\log q 
        +\frac{1}{3}\log 3
        +\frac{1}{\sqrt{q}}\left( 
        \frac{5}{3}\log 2
        +\log 5
        +\frac{4}{3}\log \pi -\frac{4}{3}\right),
\end{equation}
where $\frac{1}{3}\log 3=0.3662\cdots$ and the error term depends only on $q$. Setting
\begin{equation}
C_{\textrm{even}}(q)= \frac{1}{3}\log 3+\frac{1}{\sqrt{q}}\left( 
        \frac{5}{3}\log 2
        +\log 5
        +\frac{4}{3}\log \pi -\frac{4}{3}\right), 
\end{equation}
it follows that 
\begin{equation}
\left|L(1, \chi)\right|\leq \tfrac{1}{3}\log q + C_{\textrm{even}}(q).
\end{equation}
\item In the case of odd characters, we have: 
\begin{equation*}
\sum_{\substack{ n\geq 1 \\ (n,3)=1}}\frac{1-F_{4}(\delta n)}{n}
=
       \sum_{n\geq 1}\frac{1-F_{4}(\delta n)}{n}
 -
       \sum_{n\geq 1}\frac{1-F_{4}(3\delta n)}{3n}
\end{equation*}
where $F_4$ is defined in Proposition~\ref{pro1}. Thanks to Lemma~\ref{ramare5}, we get
\begin{multline*}
 \sum_{\substack{  n\geq 1\\ (n,3)=1}}
        \frac{1-F_{4}(\delta n)}{n}=
         -\log \delta -\log (2\pi)+\frac{3}{2}
        +2\int_{0}^{1}(1-t)\log \left|\frac{\pi \delta t}{\sin (\pi \delta t)}\right|\, dt
       \\
       +\frac{1}{3}\log \delta
        +\frac{1}{3}\log (2\pi)
        -\frac{1}{2}
        +\frac{1}{3}\log 3
        -\frac{2}{3}\int_{0}^{1}(1-t)\log \left|\frac{3\pi \delta t}{\sin (3\pi \delta t)}\right|\, dt.
\end{multline*}
 It follows that
   \begin{multline*}
   \sum_{\substack{  n\geq 1\\ (n,3)=1}}
         \frac{1-F_{4}(\delta n)}{n}
=
         -\frac{2}{3}\log \delta 
         -\frac{2}{3}\log (2\pi)+1+\frac{1}{3}\log 3 
         \\+2\int_{0}^{1}\left(1-t\right)\left(\log \left|\frac{\pi \delta t}{\sin (\pi \delta t)}\right|-\frac{1}{3}\log \left|\frac{3\pi \delta t}{\sin (3\pi \delta t)}\right|\right)\, dt.
 \end{multline*}
Now, we use Proposition~\ref{pro1} and Lemma~\ref{lem42} to obtain that
\begin{equation}
\label{equ81}
\left|L(1, \chi)\right|\leq 
-\frac{2}{3}\log \delta 
        -\frac{2}{3}\log (2\pi)+1+\frac{1}{3}\log 3
        +\frac{\pi^3 \delta ^2}{18}+\frac{2\pi^2\delta^2}{27}
        +\frac{\pi}{\sqrt{q}}
        \sum_{\substack{ 1\leq m\leq \delta q\\ (m,3)=1}}
\left(\frac{m}{\delta q}-1\right)^2.
\end{equation}
 For the last sum above, we use Lemma~\ref{lem41} to find that 
 \begin{eqnarray*}
         \sum_{\substack{1\leq m\leq \delta q \\(m,3)=1}}\left(\frac{m}{\delta q}-1\right)^2
&=&
        \sum_{\substack{ m=3k+1\\  1\leq m\leq \delta q}}\left(\frac{m}{\delta q}-1\right)^2
 +
        \sum_{ \substack{m=3k+2\\ 2\leq m\leq \delta q}}\left(\frac{m}{\delta q}-1\right)^2 
\\&\leq &                                 
                 \frac{2\delta q}{9} 
                 -\frac{14}{\delta^2q^2}
                 +\frac{14}{3\delta q}.
 \end{eqnarray*}
Then, Eq~\eqref{equ81} becomes 
\begin{equation*}
 \left|L(1, \chi)\right|\leq 
 -\frac{2}{3}\log \delta 
         -\frac{2}{3}\log (2\pi)+1+\frac{1}{3}\log 3
         +\frac{\pi^3 \delta ^2}{18}
         +\frac{2\pi ^2\delta ^2}{27}
         +\frac{2\pi \delta \sqrt{q}}{9}
         +\frac{\pi}{\delta q\sqrt{q}}\left(\frac{14}{3}-\frac{14}{\delta q}\right).
\end{equation*}
The choice $ \delta =3/ \left(\pi \sqrt{q} \right)$ yields
  \begin{equation}
  \label{equ85}
  \left|L(1, \chi)\right|\leq 
           \frac{1}{3}\log q
          -\frac{1}{3}\log (12)
          +\frac{5}{3}
          +\frac{1}{q}\left(\frac{\pi}{2}+\frac{2}{3}+\frac{14\pi^2}{9}-\frac{14\pi^3}{9\sqrt{q}}\right),
\end{equation}
with $\frac{5}{3}-\frac{1}{3}\log (12)=0.8383\cdots$. Setting 
\begin{equation}
C_{\textrm{odd}}(q)=\frac{5}{3}-\frac{1}{3}\log (12)+\frac{1}{q}\left(\frac{\pi}{2}+\frac{2}{3}+\frac{14\pi^2}{9}-\frac{14\pi^3}{9\sqrt{q}}\right),
\end{equation}
it follows that 
\begin{equation} 
\left|L(1, \chi)\right|\leq 
           \tfrac{1}{3}\log q+C_{\textrm{odd}}(q).
\end{equation}
We list below some values of $C_{\textrm{even}}(q)$ and $C_{\textrm{odd}}(q)$. 
\begin{table}[h]
\centering
\begin{tabular}{|c||c|c|c|c|c|}
 \hline
 $q$ &  $10^4$ &  $10^5$ & $10^6 $ & $2 \cdot 10^6$ & \textrm{$\infty$} \\
 \hline
 $C_{\textrm{even}} \leq $ & $0.395781$ &  $0.375558$ & $0.369162 $ & $ 0.368296 $ & $0.366205$   \\
 \hline
  $C_{\textrm{odd}} \leq $  & $0.840076$ &  $0.838539$ & $0.838382 $ & $0.838374 $ & $0.838365$ \\
  \hline
 \end{tabular}
 \end{table}
 \end{itemize}
Thus, for $q>2\cdot 10^6$, we have proved that 
  \begin{equation*}
 |L(1,\chi)|\leq \tfrac 13\log q + \left\{ \begin{array}{ll}
         0.368296
         & \textrm{when $\chi(-1)=1$,}\vspace{0.5cm}\\
         0.838374
         & \textrm{ when $\chi(-1)=-1.$}
 \end{array} \right.
 \end{equation*}
In the next section, we check that our result is valid for $2\leq q\leq 2\cdot 10^6$ using a rigorous and efficient algorithm for computing $L(1, \chi)$ for all primitive $\chi$. This completes the proof. 
\section{Numerical verification}
\label{sec4}
\subsection{The Algorithm}

In \cite{Platt2013}, the second author describes two efficient and rigorous algorithms for computing values of Dirichlet L-functions. We adapt one of those for our current purpose and for convenience we restate the key Lemma here.

\begin{lem}\label{lem:dc_dft}
For $q\in\Z\geq 3$ and given $\varphi(q)$ complex values $a(n)$ for $n\in[1,q-1]$ and $(n,q)\neq 0$, we can compute
\begin{equation}
\nonumber
\sum\limits_{n=1}^{q-1}a(n)\chi(n)
\end{equation}
for the $\varphi(q)$ characters $\chi$ in $\bigoh(\varphi(q)\log(q))$ time and $\bigoh(\varphi(q))$ space.
\begin{proof}
We construct the sum via a series of Discrete Fourier Transforms and we refer the reader to \cite{Platt2013} for the details. The existence of Fast Fourier Transform algorithms for arbitrary length inputs then gives us the claimed time complexity.
\end{proof}
\end{lem}

To exploit Lemma~\ref{lem:dc_dft} we use the following;
\begin{lem}
Let $\chi$ be any non-principal Dirichlet character of conductor $q\in\N$, $L(z,\chi)$ be its associated L-function and $\psi(\alpha)$ be the digamma function. Then
\begin{equation*}
L(1,\chi)=\frac{-1}{q}\sum\limits_{a=1}^q\chi(a)\psi\left(\frac{a}{q}\right).
\end{equation*}
\begin{proof}
For $z\neq 1$ we have the identity (see for example \S $12$ of \cite{Apostol1976})
\begin{equation}\label{eq:hur}
L(z,\chi)=q^{-z}\sum\limits_{a=1}^q \chi(a)\zeta\left(z,\frac{a}{q}\right).
\end{equation}

In addition, for any non-principal character $\chi$ we have
\begin{equation*}
\sum\limits_{a\;\textrm{mod}\;q} \chi(a)=0.
\end{equation*}
Thus for $\chi$ non-principal and for any complex constant $C$, we can replace Equation~\ref{eq:hur} with
\begin{equation*}
L(z,\chi)=q^{-z}\sum\limits_{a=1}^q \chi(a)\left(\zeta\left(z,\frac{a}{q}\right)-C\right).
\end{equation*}
In particular, again for $\chi$ non-principal, we have
\begin{equation*}
\begin{aligned}
L(1,\chi)&=\lim\limits_{z\rightarrow 1\downarrow} \sum\limits_{a=1}^q q^{-z}\chi(a)\left(\zeta\left(z,\frac{a}{q}\right)-\zeta(z)\right)\\
&=\sum\limits_{a=1}^q \frac{1}{q}\chi(a) \sum\limits_{n=1}^\infty\frac{1-\frac{a}{q}}{n(n+\frac{a}{q}-1)}.
\end{aligned}
\end{equation*} 
Finally we have the series representation of the digamma function (see $6.3.16$ of \cite{abramowitz1964}) valid for $z\neq -1,-2,\ldots$
\begin{equation*}
\psi(1+z)=-\gamma+\sum\limits_{n=1}^\infty \frac{z}{n(n+z)}
\end{equation*}
and the result follows.
\end{proof}
\end{lem}

\subsection{The Computation}

We implemented the above algorithm using the C++ programming language. To avoid potential problems with the propagation of rounding and truncation errors, we used the second author's double precision interval arithmetic package throughout. This in turn exploits the work of Lambov \cite{Lambov2008} and the CRLIBM package \cite{Muller2010}.

We ran the code on a single node of the University of Bristol's Bluecrystal cluster~\cite{ACRC2009} using all $8$ cores. The elapsed time was $69$ hours and there were no exceptions to Theorem~\ref{Thm4} over the $115,492,010,081$ primitive L-functions checked. 

At Figure~\ref{fig:even_max} we plot the maximum value of $|L(1,\chi)|-\frac{1}{3}\log q$ achieved over all even characters for each conductor $q\equiv 0\mod 3$ with $q\in[3,100000]$. Figure~\ref{fig:odd_max} shows the same information for odd characters. The reference lines indicate the bounds of Theorem~\ref{Thm4}.

The ``banding'' observed appears to be driven by the divisibility properties of the conductor. For example, the lower of the two main bands seen in each figure comprises those $q$ divisible by $12$. The largest value seen for even characters was at $q=249$ where $|L(1,\chi)|-\frac{1}{3}\log(q)<0.271789$ and for odd characters the maximum observed was at $q=111$ where $|L(1,\chi)|-\frac{1}{3}\log(q)<0.815651$. We conjecture that these upper bounds will hold for all $q$, not just those below $2\cdot 10^6$. 

\begin{figure}[h]
\centering
\fbox{\includegraphics[width=0.89\linewidth]{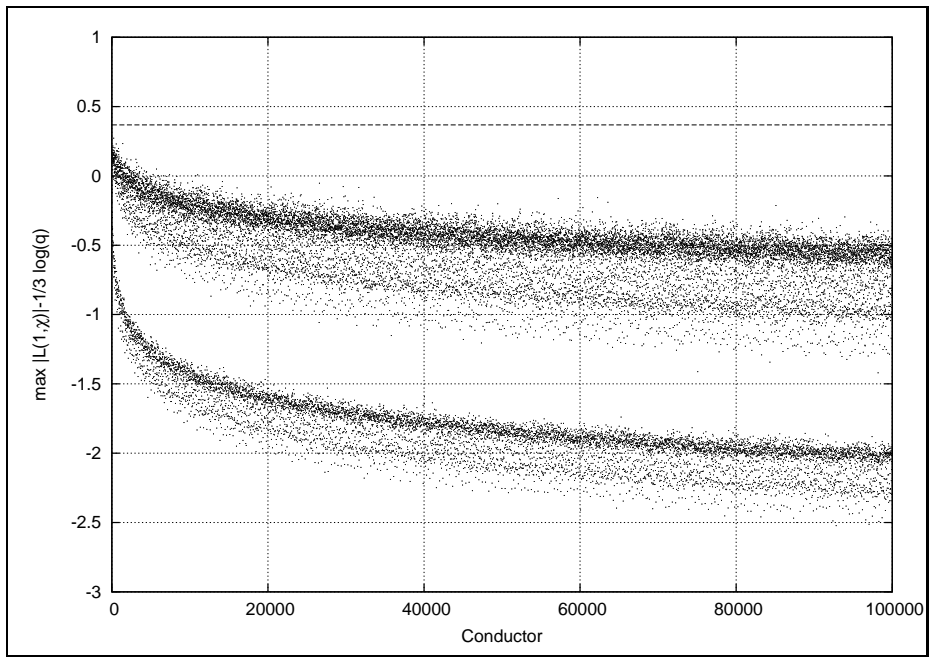}}
\caption{See text.}
\label{fig:even_max}
\end{figure}

\begin{figure}[h]
\centering
\fbox{\includegraphics[width=0.89\linewidth]{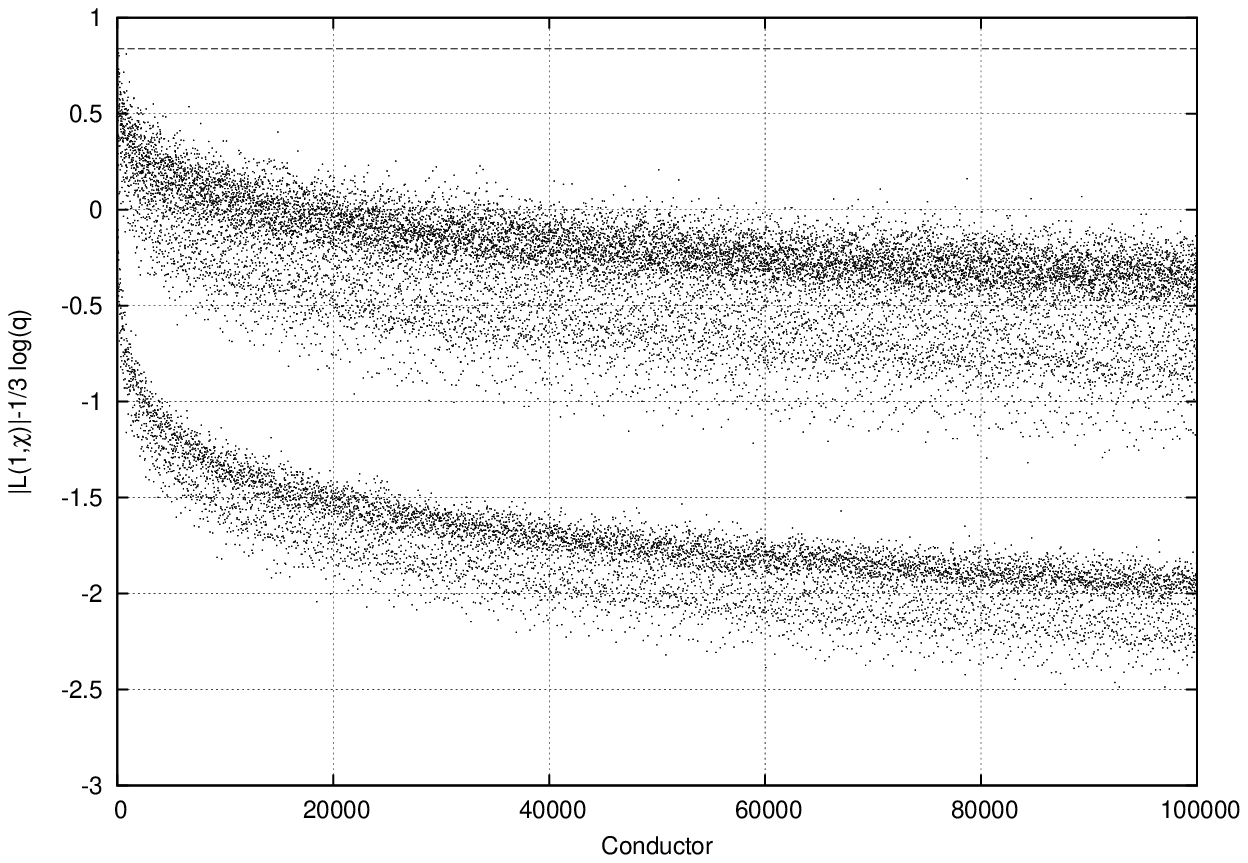}}
\caption{See text.}
\label{fig:odd_max}
\end{figure}


\bibliographystyle{amsplain}
\bibliography{davebib5}{}




         
     
     


\end{document}